\documentclass[11pt, psamsfonts]{amsart}
\usepackage{amssymb, amsmath, amsthm}
\usepackage{graphicx} 
\usepackage[final, hypertex]{hyperref}
\usepackage[a4paper, centering]{geometry}
\newtheorem{theorem}{Theorem}[section]

\newtheorem{lemma}[theorem]{Lemma}

\theoremstyle{definition}

\theoremstyle{remark}
\newtheorem{rk}[theorem]{Remark}

\newtheorem{claim}[theorem]{{\it Claim}}
\newenvironment{claim-b}
   {\begin{claim}\rm}
   {\end{claim}}

\long\def\elimina#1{}

\def\R{\mathbb{R}}

\def\Lip{\text{Lip}}

\def\pscal#1#2{#1\cdot #2}

\def\haus{\mathcal{H}^{1}}

\def\distb#1{d_{#1}}
\def\dg{\distb{\Gamma}}
\def\dist{d}
\def\uf{u_f}
\def\vf{v_f}

\def\Lipu{\Lip^1(\Omega)}
\def\Lipug{\Lip^1_{\Gamma}(\Omega)}
\def\proj{\Pi_{\Gamma}}
\def\proje{\Pi^{\epsilon}}
\def\norm{\nu}
\def\curv{\kappa}
\def\Gie{G_i^{\epsilon}}
\def\Gaie{\Gamma_i^{\epsilon}}
\def\nie{\nu^{\epsilon}}
\def\Phie{\Phi_i^{\epsilon}}
\def\le{l}

\def\rid{\mathcal{R}}
\def\Cc{C^{\infty}_c(\R^2\setminus\Gamma)}

\def\Gte{\widetilde{\Gamma}^{\epsilon}}
\def\Gtie{\widetilde{\Gamma}_i^{\epsilon}}
\def\ye{y_{\epsilon}}
\def\ke{\kappa^{\epsilon}}
\def\oray#1{{]\!]#1[\![}}
\def\cray#1{{[\![#1]\!]}}

\DeclareMathOperator{\dive}{div}
\DeclareMathOperator{\spt}{supp}
\DeclareMathOperator{\meas}{\mathcal{L}^2}

\begin{document}
\title[Sandpiles on flat tables with walls]%
{An existence result for the sandpile problem \\ on flat tables with walls}

\author[G.~Crasta]{Graziano Crasta}
\address{Dipartimento di Matematica ``G.\ Castelnuovo'', Univ.\ di Roma I\\
P.le A.\ Moro 2 -- 00185 Roma (Italy)}
\email[Graziano Crasta]{crasta@mat.uniroma1.it}

\author[S.~Finzi Vita]{Stefano Finzi Vita}
\email[Stefano Finzi Vita]{finzi@mat.uniroma1.it}

\date{\today}


\begin{abstract}
We derive an existence result for solutions of a differential system which
characterizes the equilibria of a particular model in granular matter theory, the
so-called partially open table problem for growing sandpiles.
Such result generalizes
a recent theorem of \cite{CaCa} established for the totally open table problem.
Here,
due to the presence of walls at the boundary, the surface flow density at the
equilibrium may result no more continuous nor bounded, and its explicit mathematical
characterization is obtained by domain decomposition techniques.
At the same time we
show how these solutions can be numerically computed as stationary solutions of a
dynamical two-layer model for growing sandpiles and we present the results of some
simulations.
\end{abstract}

\maketitle

\section{Introduction}

In the last years  an increasing attention has been devoted towards the
study of differential models in granular matter theory
(see, e.g., \cite{ArTs} for an overview of different theoretical approaches and models).
This field of research, which
is of course of strong relevance in the applications, has also been the source of
many new and challenging problems in the theory of partial differential equations
(see, e.g., \cite{AEW,CaCa,EGan,Pr}).

In this paper we deal with the rather simple phenomenon of the
evolution of a sandpile created by pouring dry matter on a
flat bounded table.
In such a model the table is represented by a bounded domain
$\Omega\subset\R^2$,
and the time-independent vertical matter source by
a nonnegative function $f\in L^1(\Omega)$.
Recently, Hadeler and Kuttler \cite{HK}
proposed a new model, extending the ones studied in
\cite{BCRE} and \cite{BdG},
where the description of the heap evolution is based on the
observation that granular matter forms heaps and slopes
(the so-called standing layer), while small amounts of matter
move down along the slopes, forming the so-called
rolling layer.
We also mention \cite{Pr}, where Prigozhin has studied, both from the
theoretical and numerical points of view, a degenerate parabolic problem and its
equivalent formulation as a variational inequality,
and \cite{AEW},
where a similar approach has been used
for growing sandpiles on the whole plane.
It is worth to remark that the two different dynamical models of \cite{HK} and \cite{Pr}
(see for example \cite{PrZ} for a comparison between them)
have theoretically the same set of admissible equilibria.

Let us denote by $u(t,x)$ and $v(t,x)$, $t\geq 0$, $x\in\Omega$,
respectively the heights of the standing
and rolling layers.
Neglecting wind effects,
the local dynamics depends on the local slope $|Du|$ and on the local
density of rolling matter $v$ (that we shall call
\textsl{transport density}).
For stability reasons, at any equilibrium configuration the local slope $|Du|$ cannot exceed a fixed constant
(the \textsl{critical slope}), that we normalize to the value $1$, and it must be maximal where
transport occurs (that is, where $v>0$).
Moreover, we assume that the matter falls down the table when
the base of the heap touches a portion $\Gamma$ of the boundary of $\Omega$.
{}From the model point of view we are thus assuming that
on $\partial\Omega\setminus\Gamma$ we have a (arbitrarily high) vertical wall, while on
$\Gamma$ the table is ``open".
In the following, we shall refer to the \textsl{open table problem} in the
case $\Gamma = \partial\Omega$,
whereas in the \textsl{partially open table problem}, $\Gamma$ will be a
nonempty closed subset of $\partial\Omega$.

The dynamical model proposed in \cite{HK} deals with the
open table problem, but it can be extended to the
partially open table problem in the following way:
\begin{equation}\label{f:HK}
\left\{\begin{array}{ll}
\partial_t v= \dive (v\, Du) - (1-|Du|)v + f \ &\textrm{in } [0,\infty)\times\Omega\\
\partial_t u = (1-|Du|)v \quad
&\textrm{in }[0,\infty)\times\Omega  \\
u(0,\cdot)=v(0,\cdot)=0 \quad &\textrm{in }\Omega\\
u=0 \textrm{ on } \Gamma\ ,\quad v\dfrac{\partial u}{\partial\norm}=0 \textrm{ on }
\partial\Omega\backslash\Gamma\,,
\end{array} \right.
\end{equation}
where $\frac{\partial u}{\partial\norm}$ denotes the normal derivative of $u$.
The nonlinear term which appears in the previous equations with opposite signs
represents the exchange term between the two layers during the growth process.
Before the equilibrium has been reached,
a partial surface flow is allowed also at sub-critical slopes.
As far as we know, a rigorous theory for this model is not known,
and its equilibrium configurations  have not been characterized in the general case.
For the open table case, a finite difference scheme has been proposed in \cite{FFV},
which offers at the same time a tool for the numerical
description of stationary solutions.

{}From the physical considerations above,
an equilibrium configuration $(u,v)$ for (\ref{f:HK}), with $u,v$
nonnegative functions in $\overline{\Omega}$,
must satisfy
\begin{equation}\label{f:prob2}
\begin{cases}
-\dive(v\, Du) = f &\text{in $\Omega$},\\
|Du| \leq 1        &\text{a.e.\ in $\Omega$},\\
|Du| = 1           &\text{in $\{v>0\}$},\\
u=0\ \text{on } \Gamma,
& v\, \dfrac{\partial u}{\partial\norm} = 0\
\text{on $\partial\Omega\setminus\Gamma$}.
\end{cases}
\end{equation}
In the case of the open table problem (i.e., $\Gamma = \partial\Omega$),
solutions of (\ref{f:prob2}) have been completely characterized
by Cannarsa and Cardaliaguet \cite{CaCa}
(see also \cite{CCCG} for an extension to higher dimensions).
More precisely,
denoting by $d_{\partial\Omega}\colon\overline{\Omega}\to\R$ the distance function from the
boundary of $\Omega$,
they proved that there exists a nonnegative continuous function
$\vf$ (with an explicit integral representation)
such that $(d_{\partial\Omega},\vf)$ is a solution of (\ref{f:prob2}),
and any other solution $(u,v)$ satisfies
$v = \vf$ in $\Omega$ and $u=d_{\partial\Omega}$ in $\{\vf > 0\}$.

\medskip
Aim of this paper is to extend these results
to the case of the partially open table problem.
As we shall see in the sequel, the presence of vertical walls
has a relevant influence on the regularity of stationary solutions.
Namely, we cannot expect the transport density $v$ to be a continuous
function as in the case of the open table problem.
This fact has several consequences.

First of all, we cannot give a pointwise meaning to the boundary
conditions in (\ref{f:prob2}).
Our choice here it to set $v$ in $L^1(\Omega)$, and to consider
a weak formulation of the problem
(see (\ref{f:prob}) below).
Another possible choice, which we do not pursue here,
could be to set $v$ in the class of functions with bounded variation,
so that the trace of $v$ on $\partial\Omega$ is well defined.
Our main result (see Theorem~\ref{t:exi} below) states that
there exists a nonnegative function $\vf\in L^1(\Omega)$,
of which we give an explicit integral representation,
such that $(\dg, \vf)$ is a solution to the weak formulation
of the problem, where $\dg$ denotes the distance function from $\Gamma$
(see (\ref{f:dg}) for its precise definition).

The second major consequence of the lack of continuity of $v$
concerns the uniqueness of solutions.
Namely, the uniqueness of the transport density $v$ for the open table
problem was proved in \cite{CaCa} using a blow-up argument,
first introduced in \cite{EGan} for the analysis of
mass transport problems in the framework of the Monge-Kantorovich theory,
that relies on the continuity of $v$.
In our opinion this argument cannot be adapted to the case
$v\in L^1$.
Nevertheless, it could be possible to prove a uniqueness result
in a restricted class of more regular functions
$\mathcal{F}\subset L^1(\Omega)$, provided that
one can show that $\vf\in\mathcal{F}$.

\medskip
A precise formulation of the existence result mentioned above
requires some notation.
The table $\Omega\subset\R^2$ will be a
Lipschitz domain, i.e.\ an open bounded connected set with Lipschitz boundary,
and the open boundary
$\Gamma\subset\partial\Omega$ will be a nonempty closed subset of $\partial\Omega$.
Let $\dist\colon\overline{\Omega}\times\overline{\Omega}\to [0,\infty)$
denote the path distance in $\overline{\Omega}$, defined by
\begin{equation}\label{f:pathd}
\dist(x,y) = \inf\{\textrm{length}(\gamma);\
\gamma\colon[0,1]\to\overline{\Omega}\
\textrm{Lipschitz path joining $x$ to $y$}\}\,,
\end{equation}
and let
\begin{equation}\label{f:Lipu}
\Lipu =
\{u\colon\overline{\Omega}\to\R;\
u\ \textrm{Lipschitz},\
u(x)-u(y)\leq \dist(x,y)\
\forall x,y\in\overline{\Omega}\}
\end{equation}
be the set of $1$-Lipschitz functions in $\overline{\Omega}$ with
respect to the path-metric $\dist$.
Let us denote by $\dg\colon\overline{\Omega}\to [0,\infty)$ the path distance
function from $\Gamma$, defined by
\begin{equation}\label{f:dg}
\dg(x) = \inf_{y\in\Gamma} \dist(x,y),\qquad
x\in \overline{\Omega}\,.
\end{equation}
It is easily seen that, if $\Gamma = \partial\Omega$, then
$\dg$ is the Euclidean distance from the boundary of $\Omega$.
Moreover, $\dg$ belongs to the space of functions
\begin{equation}\label{f:Lipug}
\Lipug =
\{u\in\Lipu;\ u = 0\ \text{on\ }\Gamma
\}\,.
\end{equation}
It is also known that $\dg$ is the maximal function among all
functions in $\Lipug$.
Thus $\dg$ is the maximal size of the standing layer.

Given a nonnegative function $f\in L^1(\Omega)$, let
$\uf\colon\overline{\Omega}\to [0,\infty)$ denote the function
defined by
\begin{equation}\label{f:uf}
\uf(x) = \max\{ G_z(x);\ z\in\spt(f)\}\qquad
x\in \overline{\Omega}\,,
\end{equation}
where, for every $x, z\in\overline{\Omega}$,
\begin{equation}\label{f:G}
G_z(x) =
\begin{cases}
\dg(z)-d(z,x), & \textrm{if $d(z,x)\leq \dg(z)$},\\
0, &\textrm{otherwise}\,,
\end{cases}
\end{equation}
and $\spt(f)$ denotes the (essential) support of $f$,
that is,
the complement in $\overline{\Omega}$ of the union of all
relatively open subsets $A\subseteq\overline{\Omega}$
such that $f=0$ a.e.\ in $A$.

It is readily seen that the graph of the positive part of $G_z$ is,
in the path metric,
the maximal cone of unitary slope, with apex in $z$,
whose base is contained in $\overline{\Omega}$
and touches $\Gamma$ at some point.
It is clear that $0\leq \uf\leq \dg$,
and that $\uf\in\Lipu$.
Moreover, since $\uf$ is the sup-envelope of all the cones $G_z$ with
$z\in\spt(f)$, it is plain that $\uf$ represents the minimal
standing layer for an equilibrium configuration with respect to the given support of the source.

We can now state the weak formulation
of problem (\ref{f:prob2}):
Find nonnegative functions
$u,v\colon\overline{\Omega}\to [0,+\infty)$
such that
\begin{equation}\label{f:prob}
\begin{cases}
u\in\Lipug,\ v\in L^1(\Omega),
&u,v\geq 0,\\
{\displaystyle
\int_{\Omega} v \pscal{Du}{D\phi} =
\int_{\Omega} f\,\phi}
&\forall\phi\in\Cc\,,\\
|Du| = 1
&\text{a.e.\ in } \{v>0\},
\end{cases}
\end{equation}
where $\Cc$ denotes the space of
$C^{\infty}$ functions $\phi\colon\R^2\to\R$ with compact
support in $\R^2\setminus\Gamma$.
In this weak formulation the boundary conditions on $u$ and $v$ are embedded in the
choice of the test functions space $\Cc$ and
in the condition $u\in\Lipug$.
It is readily seen that, if $u$ and $v$ are smooth enough, then problem
(\ref{f:prob}) is equivalent to (\ref{f:prob2}).


\medskip
The plan of the paper is the following.
In Section~\ref{s:exi} we state the main existence result,
constructing explicitly a transport density $\vf$
(see formula (\ref{f:vf}) below)
and showing that a pair $(u,\vf)$ is a solution
to problem (\ref{f:prob}) for every
$u\in\Lipug$ satisfying $\uf\leq u\leq\dg$.
Moreover, we show that no other function $u\in\Lipug$
can be the standing layer for the problem.
Section~\ref{s:proof} contains the proofs
of these results,
that are mainly based on a Change of Variables formula
of some independent interest
(see Theorem~\ref{t:cvar} below).
In Section~\ref{s:test}
we compute the explicit solution in a simple case,
which will be compared in Section~\ref{s:num} to the numerical
equilibrium solution of the dynamical model (\ref{f:HK})
obtained via some finite difference schemes.
Indeed, in this last section the specific difficulties of such a
numerical approximation are discussed in details.

\section{Existence of a solution}
\label{s:exi}

Throughout this paper, $\Omega\subset\R^2$
and $\Gamma\subset\partial\Omega$ satisfy the following assumptions.
\begin{itemize}
\item[(H1)]
$\Omega\subset\R^2$ is a Lipschitz domain, i.e.\
a nonempty open bounded connected set with Lipschitz boundary.
\item[(H2)] $\Gamma = \bigcup_{i=1}^{N} \Gamma_i$
is a nonempty closed subset of $\partial\Omega$,
with $\Gamma_1,\ldots,\Gamma_N$
connected arcs of $\partial\Omega$,
pairwise disjoint (up to the endpoints) and of class $C^{2}$.
We denote by $A_i$, $B_i$ the endpoints of the
arc $\Gamma_i$, $i=1,\ldots,N$, and by
$\Gamma^e$ the collection of all these endpoints.
\item[(H3)]
For every $x\in\overline{\Omega}$ there exists $y\in\Gamma$ such that
$\dg(x) = |x-y|$.
\end{itemize}

\begin{rk}\label{r:A}
Observe that (H2) does not prevent the intersection of two arcs
at the endpoints.
This is the case, for example, if $\Omega$ is a square and
$\Gamma_i$, $i=1,\ldots,4$, its sides.
Nevertheless, in order to simplify
the proofs, in the following we assume that the arcs
$\Gamma_i$ do not intersect at the endpoints.
The general case can be treated with minor modifications.
\end{rk}

Condition (H3) says that, for every $x\in\overline{\Omega}$,
there is a point $y\in\Gamma$ such that the closed segment $\cray{x,y}$
with endpoints $x$ and $y$
is a path of minimal length joining $x$ to $\Gamma$.

There are three relevant cases where these conditions are satisfied:
\begin{itemize}
\item[(a)] $\Gamma = \partial\Omega$, and $\Omega$ is a domain of class $C^{2}$.
\item[(b)] $\Gamma = \partial\Omega$, and $\Omega$ is a domain with piecewise
$C^{2}$ boundary.
\item[(c)] $\Omega$ is a non-empty open bounded convex set,
and $\Gamma$ satisfies condition (H2) above.
\end{itemize}

Case (a) and (b) refer to the open table problem.
They were considered respectively
in \cite{CaCa} (see also \cite{CCCG}
for an extension to $\Omega\subset\R^n$)
and \cite{Gio}
in the case of piecewise $C^{2,1}$ boundary with outer (i.e.\ ``convex'') corners.

Let $\Gamma^* = \Gamma\setminus\Gamma^e$.
For every $y\in\Gamma^*$ let $\norm(y)$ denote the inward unit
normal vector of $\partial\Omega$, and
let $\curv(y)$ denote the curvature
of $\partial\Omega$ at $y$.

For every $x\in\overline{\Omega}$ we denote by
\[
\proj(x) = \{y\in\Gamma;\ \dg(x) = d(x,y)\}
\]
the set of all projections of $x$ on $\Gamma$.
By (H3) it is clear that
\[
\proj(x) = \{y\in\Gamma;\ \dg(x) = |x-y|\}\,.
\]

For every $x\in\overline{\Omega}$ and every $y\in\proj(x)$, let us define
\begin{gather}
\le(x) = \dg(x)\cdot
\sup\{t>0;\ y+t(x-y)\in\Omega\ \text{and}\
y\in\proj(y+t(x-y))\},\label{f:lx}\\
m(x) = y + \le(x)\,\frac{x-y}{|x-y|},\label{f:mx}\\
R_x =\oray{y, m(x)}\,,
\end{gather}
where $\oray{x,y}$ denotes the segment joining $x$ to $y$
without the endpoints.
The set $R_x$ is called a \textsl{distance ray} (or \textsl{transport ray})
through $x$ and,
in general, it depends on the projection point $y$.
On the other hand, it is not difficult to prove that $m(x)$ and
$\le(x)$ do not depend on the choice of $y\in\proj(x)$.
It is readily seen that $\le(x)$ is the length of the transport ray $R_x$.
The set
\begin{equation}\label{f:ridge}
\rid = \{x\in\overline{\Omega};\ \dg(x) = \le(x)\}
\end{equation}
will be called the \textsl{extended ridge} of $\Omega$
(see Figure~\ref{fig:rid}).
It coincides with the usual definition of ridge when
$\Gamma = \partial\Omega$
(see \cite{EH}).
Finally let $\tau\colon\overline{\Omega}\to [0,\infty)$ denote the
\textsl{normal distance to the ridge}, defined by
\begin{equation}\label{f:tau}
\tau(x) = \le(x) - \dg(x),\qquad
x\in\overline{\Omega}\,.
\end{equation}
This function is clearly bounded from above by
$\max\{\dg(x);\ x\in\overline{\Omega}\}$.

\begin{figure}
\includegraphics[height=4cm]{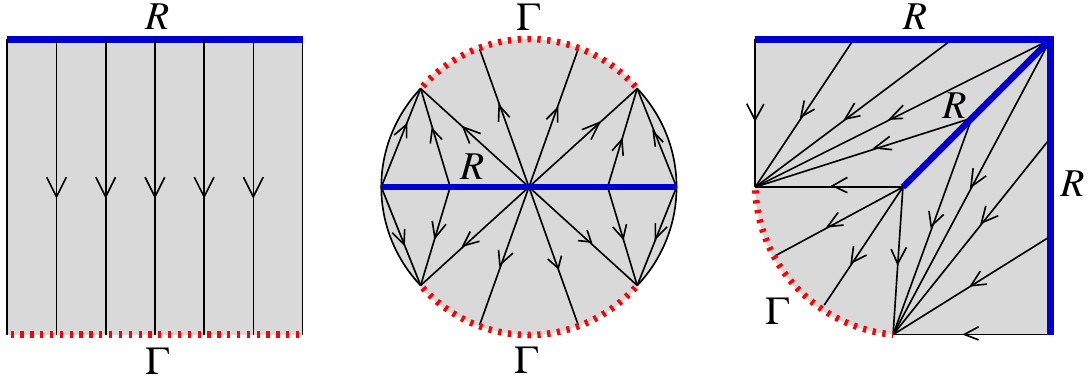}
\caption{Open boundary $\Gamma$ (dotted), extended ridge $\rid$ (bold) and transport rays}
\label{fig:rid}
\end{figure}

Let $\Omega^*\subset\Omega$ be the set of
regular points of $\dg$, that is
the set of all $x\in\Omega$
such that the projection $\proj(x)$ is a singleton.
It is well-known that $\meas(\Omega\setminus\Omega^*) = 0$.

Let $x\in\Omega^*$, let $\proj(x) = \{y\}$, and define
for every $t\in [0,\tau(x)]$
\begin{equation}\label{f:Mx}
M_x(t) =
\begin{cases}
\dfrac{\dg(x)+t}{\dg(x)},
&\textrm{if $y\in\Gamma^e$},\\
\dfrac{1-(\dg(x)+t)\curv(y)}{1-\dg(x)\curv(y)},
&\textrm{if $y\in\Gamma^*$}\,.
\end{cases}
\end{equation}

Let us define the function
\begin{equation}\label{f:vf}
\vf(x) =
\begin{cases}
\displaystyle
\int_0^{\tau(x)} f(x+t\, D\dg(x))\, M_x(t)\, dt\,,
&\textrm{if $x\in\Omega^*$},\\
0,
&\textrm{if $x\in\Omega\setminus\Omega^*$}\,.
\end{cases}
\end{equation}
We remark that,
if $\oray{y,z}$ is a transport ray, then $z\in\rid$, and
\[
\lim_{x\to z,\ x\in\oray{y,z}} \vf(x) = 0\,.
\]
In other words, the transport density $\vf$ vanishes at the end
of each transport ray.

The main theoretical contribution of this paper is the
following existence result.

\begin{theorem}[Existence]\label{t:exi}
Let $\Omega$ and $\Gamma$ satisfy (H1)-(H3).
Then the pair $(u, \vf)$ is a solution of
(\ref{f:prob}) for every
$u\in\Lipug$ satisfying $\uf\leq u\leq\dg$.
Furthermore,
if $(u,v)$ is any solution of (\ref{f:prob}),
then $\uf\leq u\leq\dg$.
\end{theorem}


It is apparent that the result above characterizes all
the possible standing layers as the functions
$u\in\Lipug$ satisfying $\uf\leq u\leq\dg$.
On the other hand, the uniqueness of the transport density
$\vf$ remains an open problem.

\section{Proof of Theorem~\ref{t:exi}}
\label{s:proof}

In what follows we shall always use the following decomposition of the set $\Omega^*$
of regular points of $\dg$.
For every $i=1,\ldots,N$, let us define the sets
\begin{equation}\label{f:decomp}
\begin{split}
\Omega_i^* & = \{x\in\Omega^*;\ \proj(x)\in\Gamma_i^*\},\\
\Omega_i^A & = \{x\in\Omega^*;\ \proj(x) = \{A_i\}\},\quad
\Omega_i^B = \{x\in\Omega^*;\ \proj(x) = \{B_i\}\}\,.
\end{split}
\end{equation}
Then (see Fig. \ref{Fig:0} for an example)
\[
\Omega^*= \bigcup_{i=1}^N \left( \Omega_i^* \cup \Omega^A_i \cup
\Omega^B_i\right)\ .
\]

\begin{figure}
\includegraphics[height=4cm]{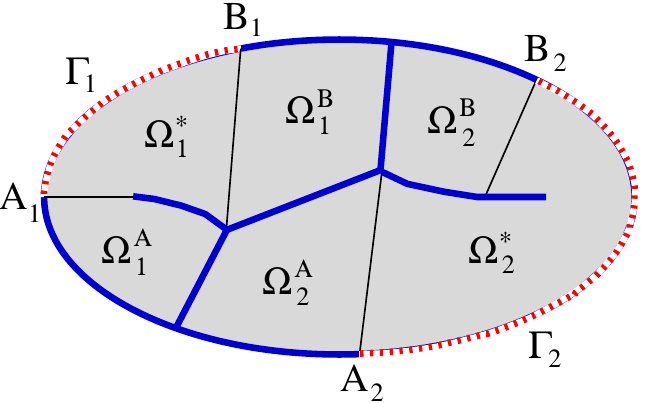}
\caption{An example of domain decomposition} \label{Fig:0}
\end{figure}

Next, we need an approximation of this decomposition
in terms of sets that can be easily parameterized
(see Theorem~\ref{t:ridge} below).
For every $\epsilon > 0$ let us define the sets
\begin{equation}\label{f:Gie}
\Gie = \left\{x\in\R^2;\ \min_{y\in\Gamma_i} |x-y| < \epsilon\right\},\quad
\Gaie = \partial\Gie,\quad
\qquad
i=1,\ldots,N.
\end{equation}
{}From Remark~\ref{r:A} and (H3), there exists $r>0$ with the following property:
for every $\epsilon\in (0,r]$, the sets
$\overline{\Gie}$ are pairwise disjoint and of class $C^{1,1}$.
Given $\epsilon\in (0,r]$ and $i=1,\ldots,N$, for every $y\in\Gaie$
let $\nie(y)$ denote the \textsl{outer} (with respect to $\Gie$)
unit normal vector to $\Gaie$ at $y$.
In the following, $\epsilon$ will always denote
a number in $(0,r]$.

Let us define the maps
\begin{equation}\label{f:Psi}
\Phie\colon\Gaie\times\R\to\R^2,\quad
\Phie(y,t) = y + (t-\epsilon)\,\nie(y),\quad
y\in\Gaie,\ t\in\R,
\quad
i=1,\ldots,N.
\end{equation}

The following lemma states that $\Omega$ can be parameterized using
the maps $\Phie$. Moreover, this parametrization is independent of the choice
of $\epsilon$.
Let us define
\[
\Gtie = \Gaie\cap\Omega,\quad
i=1,\ldots,N,\qquad
\Gamma^{\epsilon} = \bigcup_{i=1}^N \Gaie,\quad
\Gte = \bigcup_{i=1}^N \Gtie,
\]
and denote by $\proje = \Pi_{\Gamma^{\epsilon}}$
the projection operator on $\Gamma^{\epsilon}$.

\begin{lemma}\label{l:proj}
Let $x\in\Omega$ and let $y\in\proj(x)$.
If $y\in\Gamma_i$, then
for every $\epsilon\in (0,r)$
there exists a unique $\ye\in\Gaie$ such that
$\ye\in\proje(x)$ and $x = \Phie(\ye, \dg(x))$.
Moreover,
$x\in \oray{y,\ye}$, if $\dg(x) < \epsilon$, and
$\ye\in \oray{y, x}$, $\ye\in\Gtie$, if $\dg(x)>\epsilon$.
\end{lemma}

\begin{proof}
Let $\ye$ be the intersection of $\Gaie$ with
the ray starting from $y$ and passing through $x$.
{}From the definition of $\Gaie$ it is straightforward to check that
$\ye$ satisfies all the stated properties.
\end{proof}


Let $\ye\in\Gte$, and let $y\in\Gamma$ be the unique projection of $\ye$
on $\Gamma$.
Let us define
\begin{equation}\label{f:ke}
\ke(\ye) =
\begin{cases}
-1/\epsilon,
&\textrm{if $y\in\Gamma^e$},\\
\kappa(y)/(1-\epsilon\, \kappa(y)),
&\textrm{if $y\in\Gamma^*$}\,.
\end{cases}
\end{equation}
It can be checked that $\ke(\ye)$ is the curvature of
$\Gamma^{\epsilon}$ at every point $\ye$ where the curvature
is defined.
We remark that, for every $i=1,\ldots,N$, $\ke$ is defined at all but four points
of $\Gte_i$.
Let $x\in\Omega^*$, let $\proje(x) = \{\ye\}$,
and define
\begin{equation}\label{f:Mxe}
M_x^{\epsilon}(t) =
\dfrac{1-(\dg(x)+t-\epsilon)\ke(\ye)}{1-(\dg(x)-\epsilon)\ke(\ye)},
\qquad
t\in [0,\tau(x)]\,.
\end{equation}
It is easy to check that, if $\proj(x) = \{y\}$, then
\begin{equation}\label{f:Mxeug}
M_x^{\epsilon}(t) = M_x(t),
\qquad
\forall t\in [0,\tau(x)].
\end{equation}

\begin{figure}
\includegraphics[height=4cm]{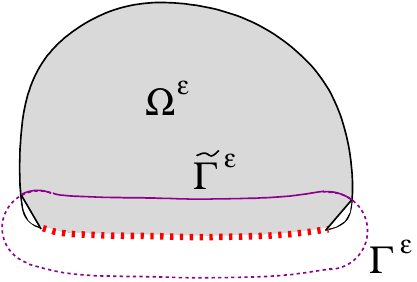}
\caption{The set $\Omega_{\epsilon}$}
\label{fig:oe}
\end{figure}

The main tool needed for the proof of Theorem~\ref{t:exi} is a
Change of Variables formula (see Theorem~\ref{t:cvar} below).
As a first step, we need the following approximation result
(see Figure~\ref{fig:oe}).

\begin{theorem}\label{t:ridge}
For every $\epsilon\in (0,r]$ and $i=1,\ldots,N$, let
\[
\Omega_i^{\epsilon} = \{\Phie(y,t);\ y\in\Gte,\ 0<t<\le(y)\},\qquad
\Omega^{\epsilon} = \bigcup_{i=1}^N \Omega_i^{\epsilon}\,.
\]
Then $\Omega^{\epsilon}\subset\Omega$, and
$\lim_{\epsilon\to 0} \meas(\Omega\setminus\Omega^{\epsilon}) = 0$.
\end{theorem}

\begin{proof}
It is not difficult to see that the set
\[
O^{\epsilon} := \bigcup_{i=1}^N \{\Phie(y,t);\ y\in\Gte,\ 0<t\leq\le(y)\}
\]
covers $\Omega$ up to a set of Lebesgue measure $N\pi\epsilon^2$, i.e.\
$\meas(\Omega\setminus O^{\epsilon}) < N\pi\epsilon^2$.
Moreover, from Lemma~\ref{l:proj} it is clear that
$\Omega^{\epsilon}\subset O^{\epsilon}\subset\Omega$.
The conclusion now follows observing that
\[
O^{\epsilon}\setminus\Omega^{\epsilon}\subseteq
D := \bigcup_{i=1}^N \{\Phie(y,\le(y));\ y\in\Gtie\}\subseteq\rid\,,
\]
and that $D$
has vanishing Lebesgue measure (see \cite{CMf}, Corollary~6.8).
\end{proof}

Let us define the map $\Phi^{\epsilon}\colon\Gte\times\R\to\R^2$ by setting
$\Phi^{\epsilon}(y,t) = \Phie(y,t)$ if $y\in\Gtie$.

\begin{theorem}[Change of Variables]\label{t:cvar}
For every $h\in L^1(\Omega)$ we have
\begin{equation}\label{f:cvar}
\begin{split}
\int_{\Omega} h(x)\, dx &=
\lim_{\epsilon\to 0}\int_{\Gte}\left[\int_{0}^{\le(y)}
h(\Phi^{\epsilon}(y,t))\, [1-(t-\epsilon)\ke(y)]\, dt
\right]\, d\haus(y)
\\ & =
\lim_{\epsilon\to 0}
\int_{\Gte}\left[\int_{-\epsilon}^{\le(y)-\epsilon}
h(y+t\nie(y))\, [1-t\ke(y)]\, dt
\right]\, d\haus(y)\,.
\end{split}
\end{equation}
\end{theorem}

\begin{proof}
{}From Theorem~\ref{t:ridge} we have that
\[
\int_{\Omega} h(x)\, dx =
\lim_{\epsilon\to 0} \int_{\Omega^{\epsilon}} h(x)\, dx =
\lim_{\epsilon\to 0}
\sum_{i=1}^N
\int_{\Omega_i^{\epsilon}} h(x)\, dx\,.
\]
For every $i=1,\ldots,N$
we have the decomposition
\[
\Omega_i^{\epsilon} = \Omega_i^{\epsilon,*}\cup \Omega_i^{\epsilon,A}
\cup \Omega_i^{\epsilon,B}\,,
\]
defined as in (\ref{f:decomp}) replacing $\Omega^*$ with $\Omega^{\epsilon}$.

Using the arguments developed in \cite[Sect.~7]{CMf},
it can be checked that
\begin{equation}\label{f:ints}
\begin{split}
\int_{\Omega_i^{\epsilon,*}} h(x)\, dx & =
\int_{\Gtie\cap\Omega_i^{\epsilon,*}}\left[\int_{0}^{\le(y)}
h(\Phi^{\epsilon}(y,t))\, [1-(t-\epsilon)\ke(y)]\, dt
\right]\, d\haus(y)\,.
\end{split}
\end{equation}
We shall now prove that
\begin{equation}\label{f:intA}
\int_{\Omega_i^{\epsilon,A}} h(x)\, dx =
\int_{\Gtie\cap\Omega_i^{\epsilon,A}}\left[\int_{0}^{\le(y)}
h(\Phi^{\epsilon}(y,t))\, [1-(t-\epsilon)\ke(y)]\, dt
\right]\, d\haus(y)\,.
\end{equation}
Let us define the region
\[
A_{\epsilon} = \{\Phi^{\epsilon}(y,t);\ y\in \Gtie\cap\Omega_i^{\epsilon,A},\
0 < t < \le(y)\}\,.
\]
It is clear that
$A_{\epsilon}\subset\Omega_i^{\epsilon,A}$.
Hence, it is enough to prove that
\begin{equation}\label{f:intAA}
\int_{A_{\epsilon}} h(x)\, dx =
\int_{\Gtie\cap\Omega_i^{\epsilon,A}}\left[\int_{0}^{\le(y)}
h(\Phi^{\epsilon}(y,t))\, [1-(t-\epsilon)\ke(y)]\, dt
\right]\, d\haus(y)\,.
\end{equation}
Observe that, by the very definition of $\ke$,
we have that $\ke(y) = -1/\epsilon$ for every
$y\in \Gtie\cap\Omega_i^A$.
Then, the integral in brackets becomes
\[
\int_{0}^{\le(y)}
h(\Phi^{\epsilon}(y,t))\, [1-(t-\epsilon)\ke(y)]\, dt
=
\int_{0}^{\le(y)}
h(\Phi^{\epsilon}(y,t))\, \frac{t}{\epsilon}\, dt\,.
\]
On the other hand,
the integral
over $A_{\epsilon}$
can be computed using
polar coordinates $(\rho, \theta)$.
We remark that
$\Gtie\cap\Omega_i^{\epsilon,A}$ is an arc of circumference of radius $\epsilon$,
so that $d\haus(y) = \epsilon\, d\theta$, hence
formula (\ref{f:intAA}) follows.

It is clear that (\ref{f:intA}) holds if
$\Omega_i^{\epsilon,A}$ is replaced by $\Omega_i^{\epsilon,B}$.
The identity (\ref{f:cvar}) then follows
from (\ref{f:ints}) and (\ref{f:intA}).
\end{proof}

\begin{lemma}\label{l:uf}
The function $\uf$ defined in (\ref{f:uf}) satisfies
$\uf = \dg$ in the set $\{\vf > 0\}$.
Moreover, if $u\in\Lipu$ is a function satisfying
$\uf\leq u\leq\dg$ in $\Omega$, then
$Du = D\dg$ in $\{\vf > 0\}$.
\end{lemma}

\begin{proof}
{}From the very definition of $\uf$ it is plain that
$\uf(x) \geq \dg(x)$ for every $x\in\spt(f)$.
On the other hand, by the maximality of $\dg$ among all
functions of $\Lipu$ vanishing on $\Gamma$, we conclude
that
\begin{equation}\label{f:uga}
\uf = \dg\ \text{in } \spt(f).
\end{equation}
Now let $x\in \{\vf > 0\}$ and let us prove that
$\uf(x) = \dg(x)$.
{}From the definition of $\vf$ we have that
$x\in\Omega^*$, that is, $\proj(x)$ is a singleton
$\{y\}$, with $y\in\Gamma$.
Let $\cray{y,z}$ be the transport ray through $x$.
Since $\vf(x) > 0$, we have that
$\oray{x,z}\cap\spt(f)\neq\emptyset$.
Let $x_0$ be a point belonging to this intersection.
{}From (\ref{f:uga}) we have that
$\uf(x_0) = \dg(x_0)$, which in turn implies that
$\uf = \dg$ along the segment
$\cray{y, x_0}$.
Since $x\in\cray{y,x_0}$, we conclude that
$\uf(x) = \dg(x)$,
and the first part of the lemma is proved.

The second part follows from \cite[Prop.~4.2]{Amb}
(see also \cite[Lemma~7.3]{Cran}
and \cite[Lemma~4.3]{CMi})
upon observing that $u=\dg$ in $\{\vf>0\}$.
\end{proof}


\begin{lemma}\label{l:max}
Let $(u, v)$ be a solution of (\ref{f:prob}). Then
\begin{equation}\label{f:max}
\int_{\Omega} f\, u = \max\left\{ \int_{\Omega} f\, w\,;\
w\in\Lipug\right\}.
\end{equation}
\end{lemma}

\begin{proof}
By a standard approximation argument we have that
\begin{equation}\label{f:varl}
\int_{\Omega} v\, \pscal{Du}{D\phi} =
\int_{\Omega} f\,\phi
\qquad
\forall \phi\colon\overline{\Omega}\to\R\
\text{Lipschitz, $\phi=0$ on $\Gamma$}.
\end{equation}
Let $w\in\Lipug$.
Since $v\geq 0$, $|Dw|\leq 1$, $|Du|\leq 1$, and $|Du| = 1$ a.e.\ in $\{v>0\}$,
we have that
\[
\int_{\Omega} v\, \pscal{Du}{(Dw-Du)}\leq 0\,.
\]
Hence we infer that
\[
\int_{\Omega} f\, u - \int_{\Omega} f\, w
\geq
\int_{\Omega} v\, \pscal{Du}{(Dw-Du)} - f(w-u) = 0,
\]
where the last equality follows from (\ref{f:varl})
with $\phi = w-u$.
Since this inequality holds for every $w\in\Lipug$,
(\ref{f:max}) follows.
\end{proof}

\begin{proof}[Proof of Theorem \ref{t:exi}]
We divide the proof into three steps.
In the first one we shall prove that $(\dg,\vf)$ is a solution to
(\ref{f:prob}).
Then, in the second step, we shall prove that also
$(u,\vf)$ is a solution for every $u\in\Lipu$
satisfying $\uf\leq u\leq\dg$.
Finally, in the last step
we shall prove that if $u\in\Lipug$ is a nonnegative function
with $\{u<\uf\}\neq\emptyset$, then for every nonnegative
function $v\in L^1(\Omega)$ the pair
$(u,v)$ is not a solution of (\ref{f:prob}).

\smallskip\noindent
\textbf{Step 1.}
We give only a sketch of the proof,
since it follows the lines of the proof of
Theorem~7.2 in \cite{CMf}.

Given $\phi\in\Cc$, we have to prove that (\ref{f:prob})
holds with $u=\dg$.
Using the Change of Variables formula (\ref{f:cvar}) we have that
\begin{equation}\label{f:pdeb}
\int_{\Omega} f\phi=
\lim_{\epsilon\to 0}\int_{\Gte}\left[\int_{0}^{\le(y)}
\phi(\Phi^{\epsilon}(y,t))\,f(\Phi^{\epsilon}(y,t))\, [1-(t-\epsilon)\ke(y)]\, dt
\right]\, d\haus(y)\,.
\end{equation}
For every fixed $y\in\Gte$, let us integrate by parts the term in brackets.
Recalling that $\Phi^{\epsilon}(y,0) = y-\epsilon\nie(y)\in\Gamma$,
we have that $\phi(\Phi^{\epsilon}(y,0)) = 0$, hence the integration by
parts gives
\[
\begin{split}
I(y)  := & \int_{0}^{\le(y)}
\phi(\Phi^{\epsilon}(y,t))\,f(\Phi^{\epsilon}(y,t))\, [1-(t-\epsilon)\ke(y)]\, dt
\\  = &
\int_{0}^{\le(y)}
\pscal{D\phi(\Phi^{\epsilon}(y,t))}{\nie(y)}\,
\int_{t}^{\le(y)}
f(\Phi^{\epsilon}(y,s))\, [1-(s-\epsilon)\ke(y)]\, ds\, dt\,.
\end{split}
\]
{}From the definition (\ref{f:vf}) of $\vf$
and (\ref{f:Mxeug}) we deduce that
\[
\vf(\Phi^{\epsilon}(y,t)) =
\int_{t}^{\le(y)}
f(\Phi^{\epsilon}(y,s))\, \frac{1-(s-\epsilon)\ke(y)}{1-(t-\epsilon)\ke(y)}\, ds\,,
\]
so that
\[
I(y) =
\int_{0}^{\le(y)}
\pscal{D\phi(\Phi^{\epsilon}(y,t))}{\nie(y)}\,
\vf(\Phi^{\epsilon}(y,t))\, [1-(t-\epsilon)\ke(y)]\, dt\,.
\]
Observe now that
$D\dg(\Phi^{\epsilon}(y,t)) = \nie(y)$ for every
$y\in\Gte$ and every $t\in (0,\le(y))$.
Substituting the last expression for $I(y)$ in (\ref{f:pdeb}) and using again
the Change of Variables formula (\ref{f:cvar}) we finally get
\[
\begin{split}
\int_{\Omega} f\phi & =
\lim_{\epsilon\to 0}\int_{\Gte}\left[\int_{0}^{\le(y)}
(\vf\, \pscal{D\phi}{D\dg})(\Phi^{\epsilon}(y,t))\,
[1-(t-\epsilon)\ke(y)]\, dt\right]\, d\haus(y)
\\ & =
\int_{\Omega} \vf\, \pscal{D\phi}{D\dg} \,.
\end{split}
\]
By the way, choosing $\phi = \dg$ (see (\ref{f:varl})) we have that
\[
\int_{\Omega} \vf = \int_{\Omega} f\dg < +\infty,
\]
hence the nonnegative function $\vf$ belongs to $L^1(\Omega)$.

\smallskip\noindent
\textbf{Step 2.}
Let $u\in\Lipu$ satisfy $\uf\leq u\leq\dg$.
{}From Lemma~\ref{l:uf} we have that
$u=\dg$ and $Du = D\dg$
in the set $\{\vf > 0\}$, hence
\[
\int_{\Omega} \vf\, \pscal{Du}{D\phi} =
\int_{\Omega} \vf\, \pscal{D\dg}{D\phi} =
\int_{\Omega} f\,\phi
\qquad
\forall \phi\in\Cc,
\]
where the last equality follows from Step~1.

\smallskip\noindent
\textbf{Step 3.}
Let $u\in\Lipug$ be a nonnegative function
with $\{u<\uf\}\neq\emptyset$.
{}From \cite{CMi}, Proposition~4.4(iii), we deduce that also the set
$\{x\in\spt(f);\ u(x) < \dg(x)\}$
is not empty and, in particular,
$u<\dg$ on a set of positive measure where $f>0$.
Since $u\leq\dg$ and $f\geq 0$, we infer that
\begin{equation}\label{f:nomax}
\int_{\Omega} f\, u < \int_{\Omega} f\, \dg\,.
\end{equation}
On the other hand
$(u,v)$ is a solution to (\ref{f:prob}),
so that (\ref{f:max}) holds,
in contradiction with (\ref{f:nomax}).
\end{proof}


\section{A test example}
\label{s:test}

In this section we describe a simple example which illustrates very well how the
presence of vertical walls on the boundary can influence the regularity of solutions
$(u,v)$ of (\ref{f:prob}). Let $\Omega=(0,1)^2$ be the unit square of $\R^2$, and
$\Gamma=\{(x,y)\in\R^2:\ 0\leq x\leq 0.5\ ;\ y=0\}$ the only open part of its
boundary. Assume $f\equiv 1$ in all of $\Omega$. From the picture in Fig. \ref{Fig:4}
we see that the sand transport rays behave differently in the two half sides of the
table: in the left-hand side they lay parallel in the direction of $\Gamma$, whereas
in the right-hand side they converge all together into the extremal point
$P=(0.5,0)$, creating a singularity.

\begin{figure}
\begin{center}
\includegraphics[height=5cm]{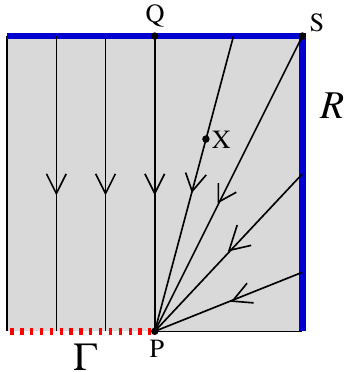}
\end{center}
\caption{The domain $\Omega$ of the example: $P$ is a singular boundary point, the
solution $v$ results discontinuous along the line $PQ$} \label{Fig:4}\end{figure}

Since $f=1$ in $\Omega$, we have that $\uf = \dg$ so that the only possible
standing layer is $u=\dg$.
The explicit computation for the solution $(\dg,\vf)$ of Theorem~\ref{t:exi}
can be done by
decomposition of the domain $\Omega$ along the segment $\overline{PQ}$.
Using polar
coordinates $(r,\theta)$ centered in $P$ (with $\theta\in[0,\frac{\pi}{2}]$) in the
right hand side, from (\ref{f:Mx}) and (\ref{f:vf}) we get in particular
\begin{equation}\label{f:sol_test}
v_f(x,y)=
\left\{{\begin{array}{ll}
1-y, &\textrm{if } x\leq 0.5, \\
\\\int_r^{l(\theta)} \frac{\rho}{r}\ d\rho\,,\ &\textrm{if } x>0.5,
\end{array}}
\right.
\end{equation}
where $l(\theta)$ denotes the length of the transport ray from $P$ to the ridge on
the wall boundary along the $\theta$ direction (see (\ref{f:lx})).
It results that
$v_f\in L^1(\Omega)$ is unbounded near $P$,
it is discontinuous along the segment $\overline{PQ}$,
and its gradient is discontinuous along the segment $\overline{PS}$.
The graph of the functions $\dg$, $\vf$ and their level lines are shown in Fig.~\ref{Fig:5}.
\begin{figure}
\begin{center}
\vspace*{-3.5cm}
\includegraphics[height=14cm,width=12cm]{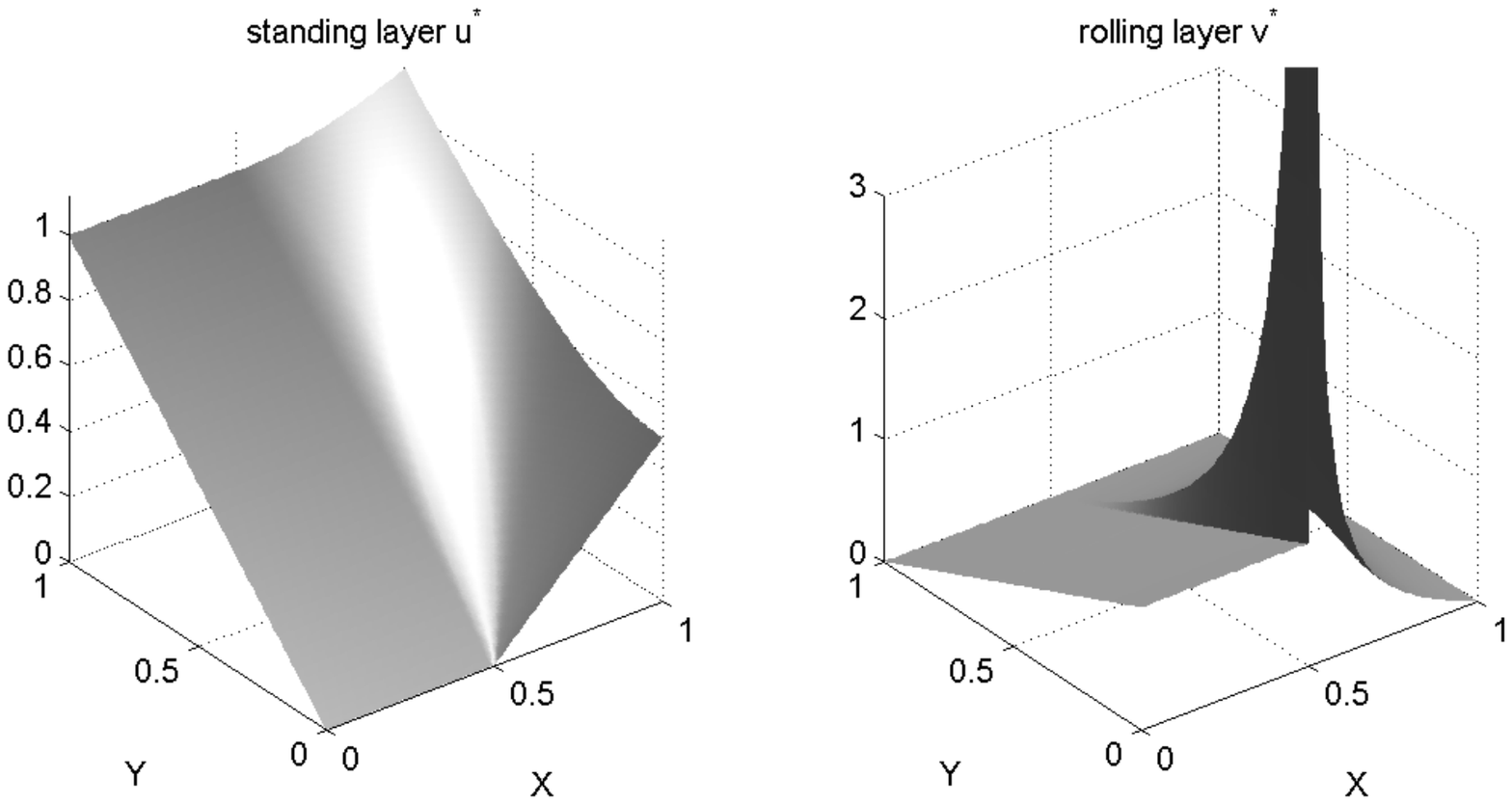}

\vspace*{-9cm}
\includegraphics[height=14cm,width=10cm]{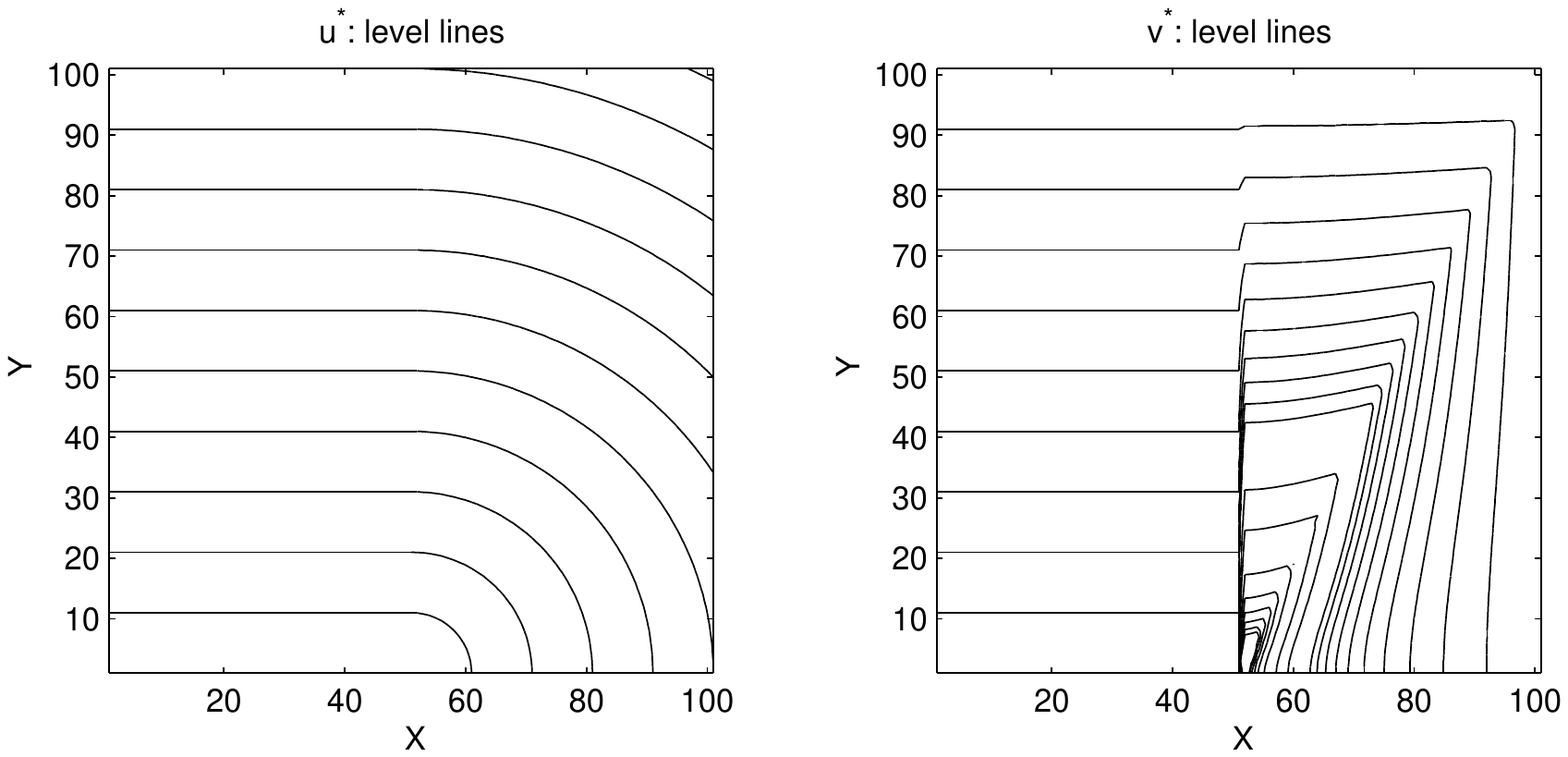}
\end{center}
\vspace*{-4.5cm}\caption{Exact solution $(d_\Gamma,v_f)$ of the test example
with level lines}
\label{Fig:5}
\end{figure}

\section{Numerical detection of stationary solutions}
\label{s:num}

In \cite{FFV} the numerical approximation of the two-layer model of \cite{HK} was
studied to simulate growing sandpiles on an open flat table. Here we have considered
the natural generalization (\ref{f:HK}) of such a model in the case of the partially open table problem,
in order to get solutions of (\ref{f:prob}) as equilibrium solutions of a
system of two evolutive partial differential equations. The extension of the finite
difference scheme introduced in \cite{FFV} to such a system is enough
straightforward. For a given discretization step $h=\Delta x$, we introduce in the
domain $\Omega$ (for simplicity, a rectangle) a uniform grid of nodes $x_{i,j}$, and
we denote as usual by $(u^n_{i,j},v^n_{i,j})$ the components of the discrete
solutions at time $t^n=n\Delta t$. Then our fully explicit finite difference scheme
can be written as

\begin{eqnarray}
&v_{i,j}^{n+1}=v_{i,j}^n + \Delta t \left[ v_{i,j}^n D^2 u^n_{i,j} +
\overline{D}v^n_{i,j}\cdot Du^n_{i,j} -
(1-|Du^n_{i,j}|)v^n_{i,j}+f_{i,j} \right], \label{2fd1}\\
&u_{i,j}^{n+1}=u^n_{i,j}+\Delta t (1-|D u^n_{i,j}|)v^n_{i,j}, \label{2fd2}\\
&u^0_{i,j}=v^0_{i,j}=0 \quad \forall i,j, \label{2fd3}\\
&u^n_{i,j}=0 \quad \textrm{ if } x_{i,j}\in\Gamma,\ \forall n, \label{2fd4}\\
&v^n_{i,j}(Du^n_{i,j}\cdot \nu_{i,j})=0 \quad \textrm{ if }
x_{i,j}\in\Omega\setminus\Gamma,\ \forall n, \label{2fd5}
\end{eqnarray}
where the discrete gradient vectors $Du^n$ and $\overline{D}v^n$ are computed
respectively, component by component, through the maxmod and the upwind finite
difference operators, and $D^2 u$ denotes the standard five-points discretization of
the Laplace operator on the grid (see \cite{FFV} for the details). What is new in
this scheme is the wall boundary condition (\ref{2fd5}), whose implementation
requires some comments. The standard way is the following: after (\ref{2fd1}) and
(\ref{2fd2}) have been applied, we look for the sign of $(Du^{n+1}\cdot \nu)$ at the
wall nodes. If it is strictly positive (as it happens for nodes which are in the
extended ridge $\rid$, that is which are starting points of a transport ray to
$\Gamma$) then $v_{i,j}^{n+1}$ is set to zero. If this is not the case, one should
modify $u_{i,j}^{n+1}$ on the boundary in order to fulfill (\ref{2fd5}). This is not
the best strategy. In fact this situation corresponds to the pathological case of
nodes belonging to boundary transport rays, that is when there exist straight portions
of the wall boundary, as in the test example of Section 4. Referring to Fig.
\ref{Fig:4}, on the west side of the square the sand flow is parallel to the boundary
(and also to the mesh in this particular case) and there is no need to impose any
boundary condition: the discrete solution $u^n$ naturally satisfies a no flux
condition at those points. Also the south portion of the wall in the example
coincides with a transport ray, but in that case the normal derivative of $u$ is
naturally negative in the boundary nodes, and it becomes zero only asymptotically in
time (at the equilibrium). Then, by continuity arguments, the best choice seems to us
simply to impose a no flux boundary condition for $v$ at those points.

The direct application of scheme (\ref{2fd1})-(\ref{2fd5}) is anyway not so
efficient, due to the numerical difficulty of handling unbounded discontinuous
solutions. In Fig. \ref{Fig:6}, the computed stationary solutions for (\ref{f:HK})
and their level lines are shown in the test example of Section 4 (compare with Fig.
\ref{Fig:5}).

\begin{figure}
\begin{center}
\vspace*{-3.5cm}
\includegraphics[height=14cm,width=12cm]{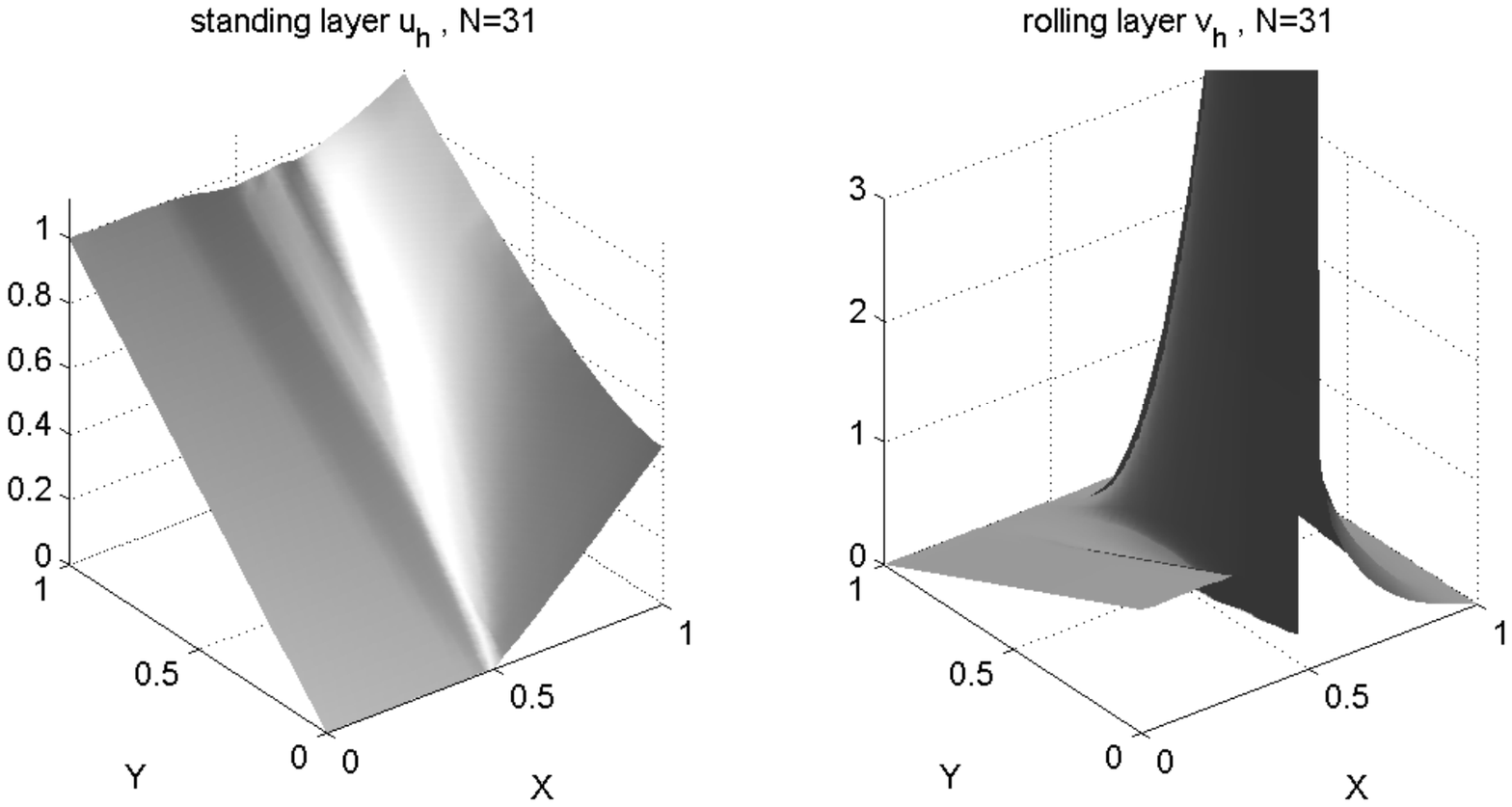}

\vspace*{-9cm}
\includegraphics[height=14cm,width=10cm]{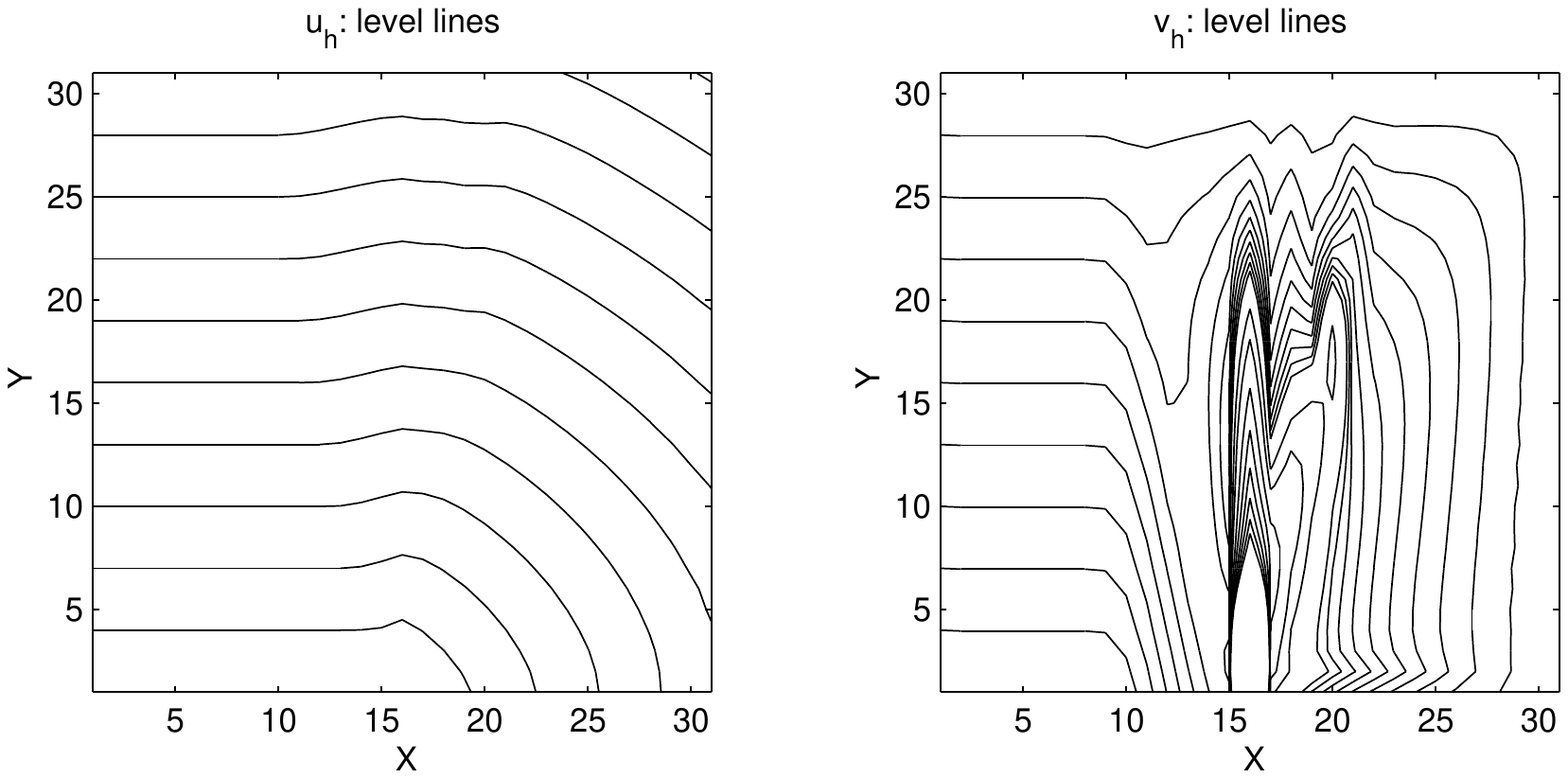}
\end{center}
\vspace*{-4.5cm}\caption{Numerical stationary solutions $u$ and $v$ of system
(\ref{f:HK}) in the test example and their level lines} \label{Fig:6}
\end{figure}

Despite the fact that the real sand flow is completely separated in the left and the
right subregions of $\Omega$, at the numerical level the flow travels through the
grid points and then it can cross the separation line.
More precisely, the transport
path for sand from a point $X$ in the right hand side should be the segment
$\overline{XP}$ in Fig. \ref{Fig:4};
on the contrary, the algorithm splits this flow along vertical and
horizontal segments connecting nodes and then part of this sand reaches the segment
$\overline{PQ}$ even far from $P$ (and from there eventually the left-hand side of
the table). That is why the simple use at the discrete level of the same
decomposition strategy adopted to characterize the stationary solutions is not able
to reduce this phenomenon. As a test we applied in fact on the same uniform grid the
scheme (\ref{2fd1})-(\ref{2fd5}) separately in the two subregions of $\Omega$, with
suitable wall boundary conditions on the cut (the $\overline{PQ}$ segment). The
results (see Fig. \ref{Fig:7}) show an evident improvement of the solutions only in
the left (that is the regular) subregion.

Better results can be expected by coupling decomposition with suitable grid
strategies. Keeping the uniform grid, the use of semi-lagrangian type schemes along
characteristics should give a better trace of the correct transport directions.
On the other hand, a different idea could be
to employ unstructured grids (and mesh refinements near the singularity regions) in
order to improve the accuracy. The discussion of these approaches will be the goal of
a forthcoming paper. The main difficulty is, anyway, that a sharp domain
decomposition requires the a priori knowledge of the ridge set, which is not in
general an easy task. For example, if we slightly modify the table in Fig.~\ref{Fig:4}
by simply opening a symmetric portion of the boundary on its northern
side ($\{(x,y):\ y=1,\ 0.5\leq x\leq 1\}$), the situation becomes completely
different: a curved internal ridge appears, and with the help of the normal
directions to the singular boundary points it subdivides the table into four distinct
flow regions (see Fig. \ref{Fig:8}, where the $v$ surface is now seen from above,
showing, in white, the ridge set profile).

\begin{figure}
\begin{center}
\vspace*{-3.5cm}
\includegraphics[height=14cm,width=12cm]{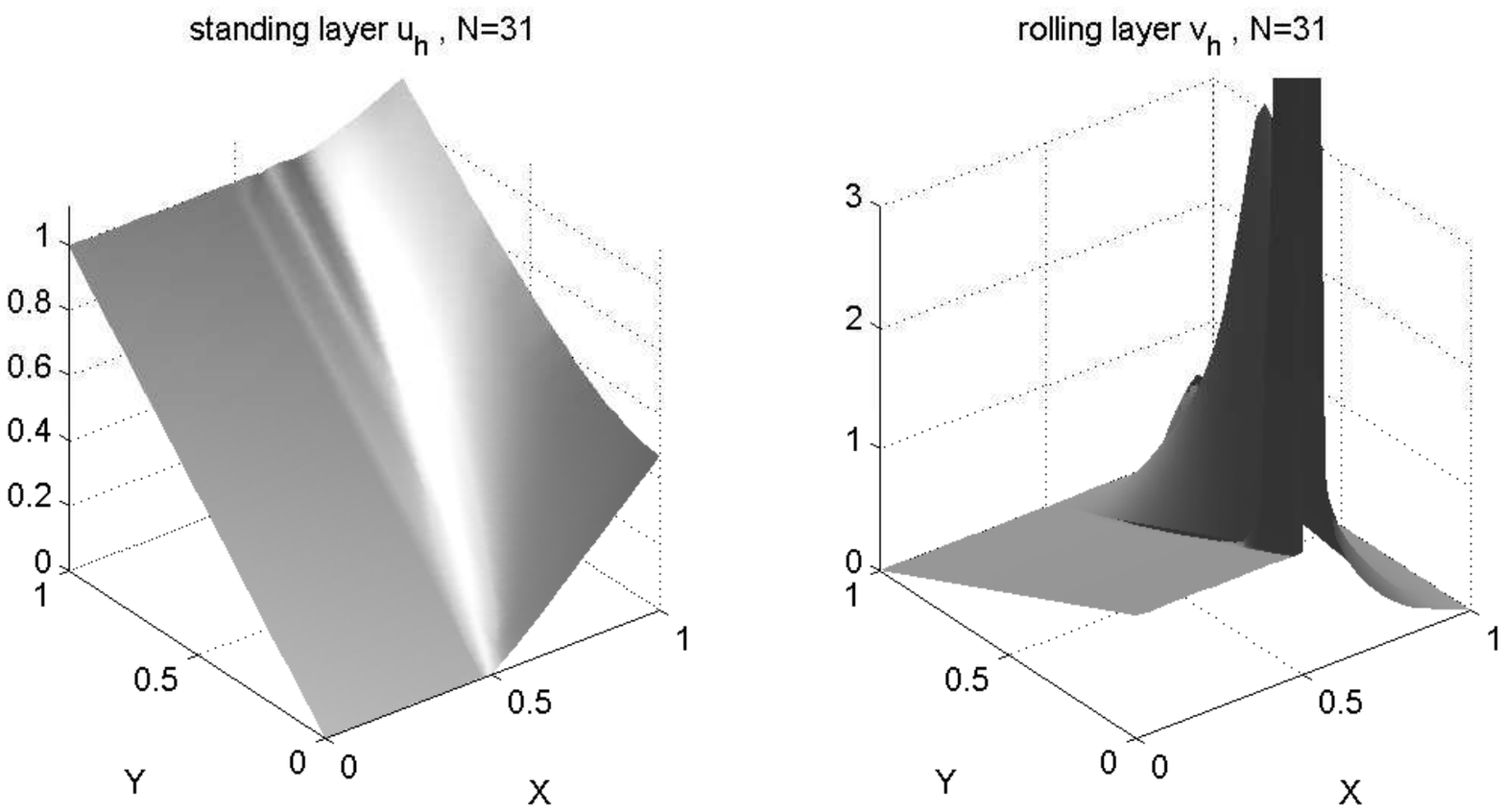}

\vspace*{-9cm}
\includegraphics[height=14cm,width=10cm]{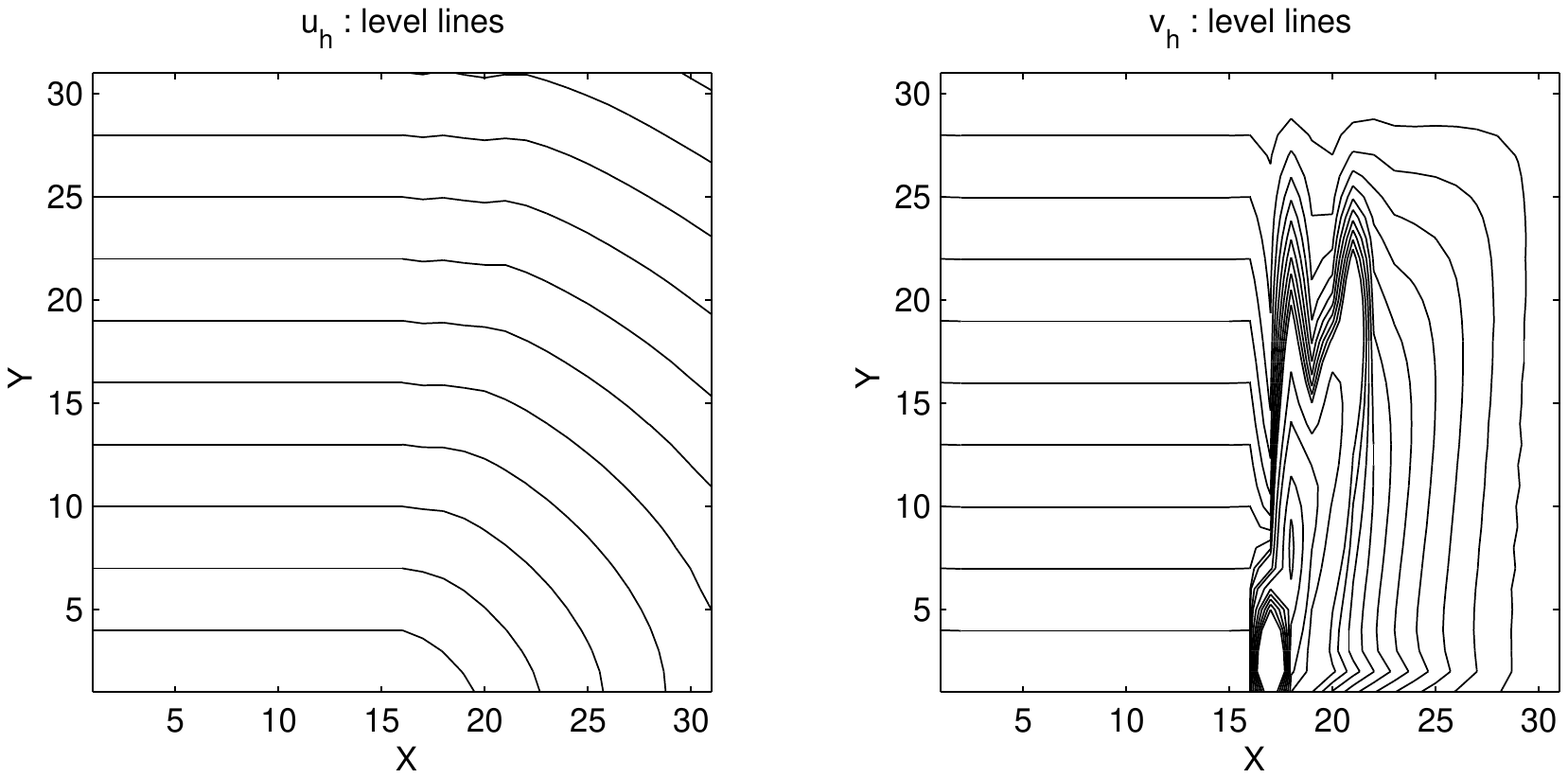}
\end{center}
\vspace*{-4.5cm}\caption{Numerical stationary solutions $u$ and $v$ of system
(\ref{f:HK}) in the test example and their level lines: solution by decomposition}
\label{Fig:7}
\end{figure}

\begin{figure}
\begin{center}
\vspace*{-3.5cm}
\includegraphics[height=14cm,width=12cm]{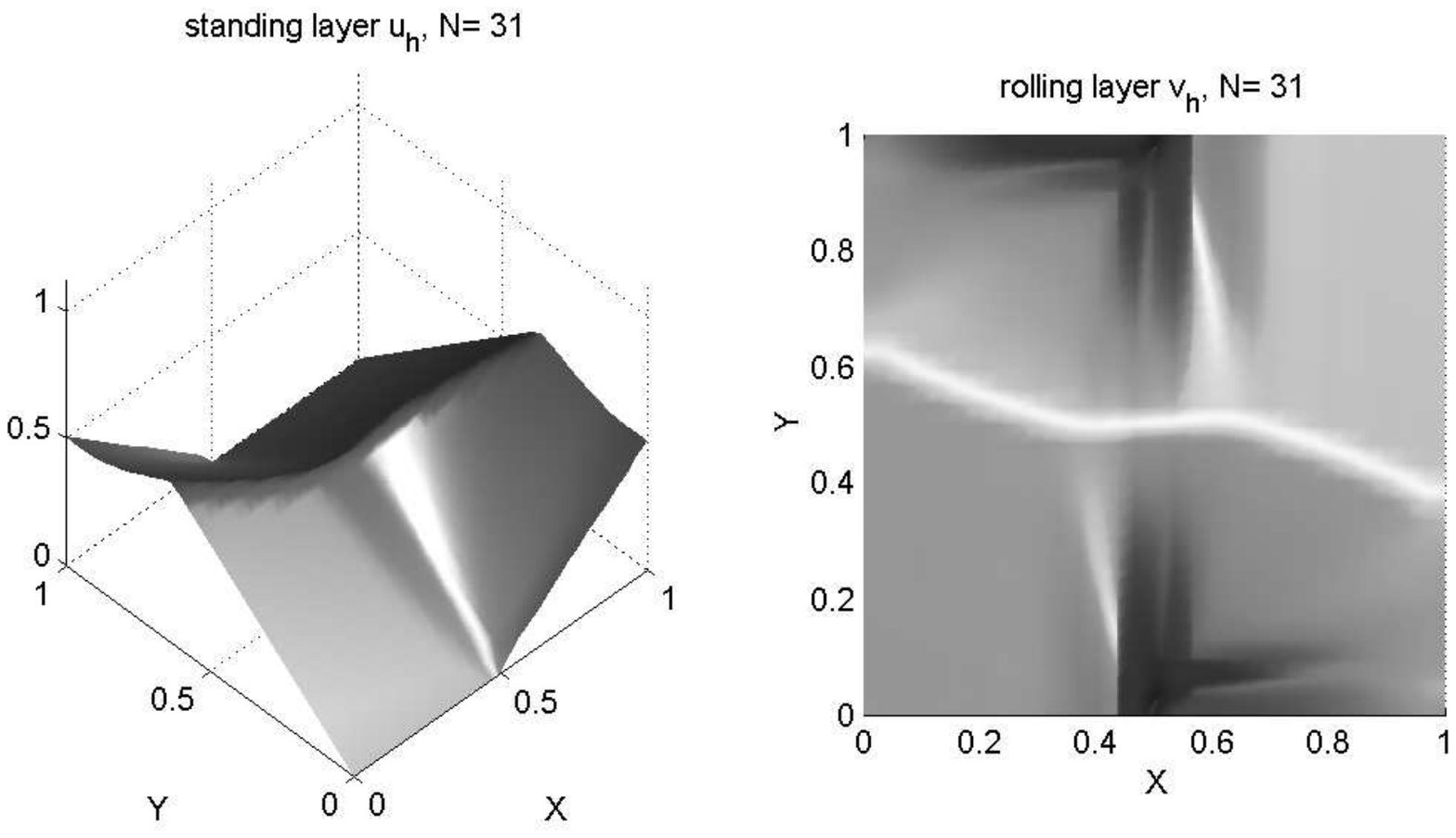}
\end{center}
\vspace*{-4.5cm}\caption{Numerical stationary solutions $u$ and $v$ (view from above)
of system (\ref{f:HK}) in the modified test example } \label{Fig:8}
\end{figure}

%

\begin{thebibliography}{10}

\bibitem{Amb}
{L.} Ambrosio, \emph{Lecture notes on optimal transport problems}, Mathematical
  Aspects of Evolving Interfaces, Lecture Notes in Math., vol. 1812,
  Springer-Verlag, Berlin/New York, 2003, pp.~1--52.

\bibitem{ArTs}
{I.S.} Aranson and {L.S.} Tsimring, \emph{Patterns and collective behavior in
  granular media: theoretical concepts}, Rev.\ Mod.\ Phys. \textbf{78} (2006),
  641--692.

\bibitem{AEW}
{G.} Aronsson, {L.\ C.} Evans, and {Y.} Wu, \emph{Fast/slow diffusion and
  growing sandpiles}, J.\ Differential Equations \textbf{131} (1996), no.~2,
  304--335.

\bibitem{BCRE}
{J.~-P.} Bouchaud, {M.~E.} Cates, {J.} Ravi~Prakash, and {S.~F.} Edwards,
  \emph{A model for the dynamics of sandpile surfaces}, J. Phys. I France
  \textbf{4} (1994), 1383--1410.

\bibitem{BdG}
{T.} Boutreux and {P.-G.} de~Gennes, \emph{Surface flows of granular mixtures,
  I. General principles and minimal model}, J.\ Phys.\ I France \textbf{6}
  (1996), 1295–--1304.

\bibitem{CaCa}
{P.} Cannarsa and {P.} Cardaliaguet, \emph{Representation of equilibrium
  solutions to the table problem for growing sandpiles}, J.\ Eur.\ Math.\ Soc.\
  (JEMS) \textbf{6} (2004), 435--464.

\bibitem{CCCG}
{P.} Cannarsa, {P.} Cardaliaguet, {G.} Crasta, and {E.} Giorgieri, \emph{A
  Boundary Value Problem for a {PDE} Model in Mass Transfer Theory:
  Representation of Solutions and Applications}, Calc.\ Var.\ Partial
  Differential Equations \textbf{24} (2005), 431--457.

\bibitem{Cran}
{M.G.} Crandall, \emph{A visit with the $\infty$-{L}aplace equation}, notes for
  CIME course, 2005.

\bibitem{CMf}
{G.} Crasta and {A.} Malusa, \emph{The distance function from the boundary in a
  {M}inkowski space}, Trans.\ Amer.\ Math.\ Soc. \textbf{359} (2007),
  5725--5759.

\bibitem{CMi}
{G.} Crasta and {A.} Malusa, \emph{A sharp uniqueness result for a class of
  variational problems solved by a distance function}, J.\ Differential
  Equations \textbf{243} (2007), 427--447.

\bibitem{EGan}
{L.C.} Evans and {W.} Gangbo, \emph{Differential equations methods for the
  {M}onge-{K}an\-to\-ro\-vich mass transfer problem}, Mem.\ Amer.\ Math.\ Soc.
  \textbf{137} (1999), no.~653.

\bibitem{EH}
{W.D.} Evans and {D.J.} Harris, \emph{Sobolev embeddings for generalized ridged
  domains}, Proc.\ London Math.\ Soc. \textbf{54} (1987), 141--175.

\bibitem{FFV}
{M.} Falcone and {S.}~Finzi Vita, \emph{A finite difference approximation of a
  two-layer system for growing sandpiles}, SIAM\ J.\ Sci.\ Comput. \textbf{28}
  (2006), 1120--1132.

\bibitem{Gio}
{E.} Giorgieri, \emph{A boundary value problem for a {PDE} model in mass
  transfer theory: representation of solutions and regularity results}, Ph.D.
  thesis, Universit\`a di Roma ``Tor Vergata'', Roma, 2004.

\bibitem{HK}
{K.P.} Hadeler and {C.} Kuttler, \emph{Dynamical models for granular matter},
  Granular\ Matter \textbf{2} (1999), 9--18.

\bibitem{Pr}
{L.} Prigozhin, \emph{Variational model of sandpile growth}, European J.\
  Appl.\ Math. \textbf{7} (1996), 225--235.

\bibitem{PrZ}
{L.} Prigozhin and {B.} Zaltzman, \emph{Two continuous models for the dynamics
  of sandpiles surface}, Phys.\ Rev.~E \textbf{63} (2001), 041505.

\end{thebibliography}

\end{document}